\documentclass[runningheads, envcountsame, a4paper,draft]{llncs}

\usepackage[utf8]{inputenc}

\usepackage{amsfonts,amsmath}
\usepackage{verbatim,amssymb,amscd}

\usepackage{pgf,tikz,calc}
\usetikzlibrary{arrows,automata,shapes}
\tikzstyle{every state}=[minimum size=12pt,inner sep=0pt]

\usepackage{thmtools}
\usepackage{thm-restate}

\usepackage{bbm}
\usepackage{bm}

\newcommand{\resp}{\textit{resp. }}
\newcommand{\ie}{\textit{i.e. }}

\DeclareMathOperator{\lcm}{lcm}

\newcommand{\dz}{\mathfrak d}
\newcommand{\aut}[1]{{\mathcal #1}}
\newcommand{\auta}{\aut{A}}

\newcommand{\mot}[1]{{\bm{#1}}}

\newcommand{\alphab}{\Sigma}
\newcommand{\auttuple}{(Q,\alphab,\delta,\rho)}
\newcommand{\pres}[1]{\langle#1\rangle}

\newcommand{\presc}[1]{\langle#1\rangle_c}
\newcommand{\sgn}[1]{\text{sgn}(#1)}
\newcommand{\sgnpi}[1]{\text{sgn}_{\pi}(#1)}
\newcommand{\ordr}[1]{\left|#1\right|}
\newcommand{\card}[1]{\left|#1\right|}

\newcommand{\ssp}{\sigma}
\newcommand{\ttp}{\tau}
\newcommand{\os}{\ordr{\ssp}}
\newcommand{\ot}{\ordr{\ttp}}
\newcommand{\N}{{\mathbb N}}
\newcommand{\Z}{{\mathbb Z}}
\newcommand{\GS}{S_k}
\newcommand{\GA}{A_k}

\newcommand{\idp}{e}
\newcommand{\dualp}[1]{\overline{#1}}
\newcommand{\GSk}[1]{S_{#1}}

\newcommand{\autpd}[2]{\raisebox{#2}{
	 \scalebox{#1}{\begin{tikzpicture}[->,>=latex]
	\node[state,black]  at (0,-1) (1)  {$1$};
	\node[] (c) [right of=1]  {};
	\node[state,black] (2) [right of=c]  {$2$};
	\path
		 (1) edge[bend left] node[above] {\LARGE{$\sigma$}} (2)
		 (2) edge[bend left] node[below] {\LARGE{$\tau$}} (1)
		 ;

	\end{tikzpicture}
	}
	}
}

\begin{document}

\mainmatter 

\title{An analogue to Dixon's theorem\\ for automaton groups}
\titlerunning{Dixon's theorem for automaton groups}

\author{Thibault Godin\thanks{The author is supported by the French
\emph{Agence Nationale pour la~Recherche},
through the Project~$\mathbf{MealyM}$ ANR-JCJC-12-JS02-012-01.}}
\authorrunning{Th. Godin}
\institute{ IRIF,
    UMR 8243 CNRS, Univ. Paris Diderot
    Paris France\\
\email{godin@liafa.univ-paris-diderot.fr}
}

\maketitle
  \abstract{Dixon's famous theorem states that the group generated by two random permutations of a finite set is generically either the whole symmetric group or the alternating group. In the context of random generation of finite groups this means that it is hopeless to wish for a uniform distribution -- or even a non-trivial one -- by drawing random permutations and looking at the generated group.\\
 Mealy automata are a powerful tool to generate groups, including all finite  groups and many interesting infinite ones, whence the idea of generating random finite groups by drawing random Mealy automata.\\
 In this paper we show that, for a special class of Mealy automata that generate only finite groups, the distribution is far from being uniform since the obtained groups are generically a semi-direct product between a direct product of alternating groups and a group generated by a tuple of transpositions.}
  \keywords{Random~Generation, Finite~Groups, Automaton~Groups, Random~Permutations.}

\newpage

\section{Introduction}
The problem of random generation of finite groups was introduced by Netto in 1882~\cite{Net}, who conjectured that two random elements of the symmetric group on $k$ elements generate either the symmetric group or the alternating group with high probability when $k$ goes to infinity. This was confirmed by Dixon in~1969~\cite{Dix}. In this paper we study a different approach for this problem, using \emph{Mealy automata}. More precisely, instead of drawing generators of a group, we draw a Mealy automaton next used to generate a group. In  specific  classes of automata we prove  analogues to Dixon theorem: the limiting probability distributions on groups is formed by atoms of total weight 1.

In all the paper we denote by $\GS$ the \emph{symmetric group} over $k$ elements, \ie the group of the bijections of the set $\{1, \hdots k\}$ (or equivalently the set of permutations on $k$ symbols equipped with the multiplication), $\os$ the \emph{order} of a permutation~$\ssp \in \GS$, and~$\sgn{\ssp}$ its \emph{signature}. Moreover we say that $\GA$ is the \emph{alternating group}, that is the group of permutations of signature 1. If $g_1,\hdots, g_i$ are elements of a group $G$ then $\pres{g_1,\hdots, g_i}$, called the \emph{group generated} by $g_1,\hdots, g_i$, is the smallest subgroup of $G$ that contains $g_1,\hdots g_i$. Finally we denote by~$\leq$ the subgroup relation, by $\rtimes$ the (inner) semidirect product, by $\pi^\rho  = \rho^{-1}\pi\rho$ the \emph{conjugate} of $\pi$ by $\rho$, by  $\idp$ the trivial permutation, and by $\lcm$ (\resp $\gcd$) the function \emph{lowest common multiple} (\resp smallest common multiple).

\bigskip

In 1969, Dixon proved  the following theorem:
\begin{theorem}[Dixon \cite{Dix}]\label{thm:dixon}
Let $\ssp$ and $\ttp$ be two random permutations in~$\GS$. Then
 \[\lim_{k \to \infty} \mathbb{P}\left(\pres{\ssp, \ttp}= \GS \text{ or } \GA \right) = 1  \:.\]
\end{theorem}
In other terms the group generated by two random permutations is generically the biggest possible. The case where the generated group is $\GA$ occurs when both  permutations have signature 1.\medskip

The asymptotic in Theorem~\ref{thm:dixon} has been gradually improved  to  $1-1/k -1/k^2-4/k^3-23/k^4-171/k^5 \hdots$\footnote{Sequence A113869 from OEIS \cite{OEIS}.} by Bovey and Williamson~\cite{BoWi},  Babai~\cite{Bab},  and  Dixon himself~\cite{Dix05}. This theorem has also been extended to any finite simple group~\cite{KaLu,LiSh}: pick two elements in a finite simple group, they generate the whole group with high probability.\medskip

This leaves open the problem of finding a suitable way to generate random finite groups. In this paper we examine a new model, that is, generating random finite groups via drawing random Mealy automata. More precisely we restrict ourselves to a class of automaton where every generated group is finite, the class of \emph{automata with cycles without exit}~\cite{Ant,Rus,KlPi}. Random generation of automata in this class has been studied by De Felice and Nicaud in~\cite{DFeNi}. \\ In Section~\ref{sec-Mealy} we recall the definitions  of Mealy automata, introduce the class of Mealy automata that we consider in the rest of the paper, and prove some properties on the (finite) groups they generate. In Section~\ref{sec-2case} we show  an analogue to Dixon's Theorem for cycle automata with 2 states, namely that, generically  \[ \pres{\autpd{0.5}{-.9em}}=  \begin{cases} 

\text{ either } \GS \times \GS \:, \\ 

 \text{ or }  (\GA\times \GA)\rtimes \pres{(\pi,\pi)}, \text{ with }\: \pi^2 = \idp \text{ and } \pi \neq \idp\\
 \text{ or }  \GA\times \GA \:.
\end{cases} \] In Section~\ref{sec-genecase} we extend this theorem to   cyclic automata with any number of states. Section~\ref{sec:ccl} is dedicated to the conclusion and some perspectives.
Due to space constraints, several proofs are omitted but can be found in the appendix.

\section{Mealy Automata}\label{sec-Mealy}
Mealy automata have been introduced by Mealy in~\cite{Mea} but have  been widely used in (semi)group theory since Glushkov~\cite{Glu}. They have given numerous interesting groups, the most famous probably being the Grigorchuk group, that is an infinite torsion group with intermediate growth, solving both the Burnside problem and the Milnor problem~\cite{Gri:burnside,Gri:milnor}, along with many others. For a more complete introduction to the topic  we refer the reader to the survey of Nekrashevych~\cite{Nek} or to the chapter of Bartholdi and Silva~\cite{BaSi}.\bigskip

A  \emph{Mealy automaton} is a complete deterministic letter-to-letter transducer $\auta = \auttuple$ where $Q$ and $\alphab$ are finite sets respectively called the \emph{stateset} and the \emph{alphabet}, and $\delta= (\delta_i : Q \to Q)_{i \in \alphab}$, $\rho = (\rho_q : \alphab \to \alphab )_{q \in Q}$ are respectively called the \emph{transition} and \emph{production} functions. These functions can be extended to words as follows:  see~\(\aut{A}\) as an automaton with  input and  output tapes, thus
defining mappings from input words over~$\Sigma$ to output words
over~$\Sigma$.
Formally, for~\(q\in Q\), the map~$\rho_q\colon\Sigma^* \rightarrow \Sigma^*$,
extending~$\rho_q\colon\Sigma \rightarrow \Sigma$, is defined recursively by:
\begin{equation}
\forall i \in \Sigma, \ \forall \mot{s} \in \Sigma^*, \qquad
\rho_q(i\mot{s}) = \rho_q(i)\rho_{\delta_i(q)}(\mot{s}) \:.
\end{equation}\\
We can also extend the map $\rho$ to words of states $\mot{u} \in Q^*$ by composing the production functions associated with the letters of $\mot{u}$:

\begin{equation}
 \forall q \in Q, \ \forall \mot{u} \in Q^*, \qquad
\rho_{q\mot{u}} = \rho_{\mot{u}}\circ \rho_{q}\:.
\end{equation}\\
A Mealy automaton is said to be \emph{invertible} whenever $\rho_q$ is a permutation of the alphabet for every $q \in Q$. It is called \emph{reversible} whenever $\delta_i$ is a permutation of the stateset for every $i \in \alphab$. Moreover an automaton is said to be \emph{bireversible} whenever it is reversible (\ie every input letter induces a permutation of the stateset) and  every output letter induces a permutation of the stateset.

 Examples of such automata are depicted in Figures~\ref{fig:mealy1} and~\ref{fig:mealy2}.
\bigskip
\begin{figure}[ht]%
\centering
\begin{minipage}{.41\textwidth}%
	\centering
	\begin{tikzpicture}[->,>=latex,node distance=15mm]
	\node[state] (1) {\(x\)};
	\node[state] (2) [right of=1] {\(y\)};

	\path 
      (1) edge[loop left] node[left]{\(\begin{array}{c} 1|2\\2|1\end{array}\)} (1)
      (2) edge node[below]{\(\begin{array}{c} 1|1\\2|2\end{array}\)} (1)
;
	\end{tikzpicture}
\caption{An invertible non-reversible Mealy automaton (generating $K_4~=~\Z/2\Z \times \Z/2\Z$).}%
\label{fig:mealy1}
\end{minipage}%
\hspace*{2mm}
\begin{minipage}{.43\textwidth}%
	\centering
	\begin{tikzpicture}[->,>=latex,node distance=15mm]
	\node[state] (1) {\(x\)};
	\node[state] (2) [right of=1] {\(y\)};

	\path 
      (1) edge[loop left] node[left]{\(\begin{array}{c} 1|1\end{array}\)} (1)
      (2) edge[bend left]  node[below]{\(\begin{array}{c} 2|1\end{array}\)} (1)
      (1) edge[bend left]  node[above]{\(\begin{array}{c} 2|2\end{array}\)} (2)
      (2) edge[loop right]  node[right]{\(\begin{array}{c} 1|1  \end{array}\)} (2)
;
	\end{tikzpicture}
\caption{A reversible non-invertible Mealy automaton (generating an infinite semigroup).}%
\label{fig:mealy2}%
\end{minipage}%
\end{figure}
\medskip

The production functions $\rho_q : \Sigma^* \to \Sigma^* $ of an automaton $\auta$ generate a semigroup.
Whenever $\auta$ is   invertible, one can define the  \emph{group generated by $\auta$}: \[\pres{\auta}:= \pres{\rho_q | q \in Q} = \pres{\rho_{\mot{u}} : \Sigma^* \to \Sigma^* | \mot{u} \in Q^*}\:.\] 
\\
The problem of deciding whether an  automaton semigroup is finite was found undecidable by Gillibert~\cite{Gil}. For automaton groups the problem is still open. However some classes of automata are known to generate only finite or infinite groups. For instance automata that are invertible and  reversible, but not bireversible generate only infinite groups~\cite{AKLMP12}; while Antonenko~\cite{Ant} and Russyev~\cite{Rus} independently proved that automata with cycle without exit, \ie automata where the underlying digraph consists in a directed  graph where each cycle is a dead end (or equivalently a directed acyclic graph with eventually cycles and loops on the leaves), generate only finite groups, regardless of the production functions. This class is maximal in the sense that, for any automaton out of this class, there exist   production functions such that the generated group is infinite~\cite{KlPi}.\\
 In this paper we focus on the simplest of these later automata, namely those where the underlying digraph consists in a single cycle. We  call the later \emph{cyclic automata}, and draw the transitions $x \xrightarrow{~\rho_x~} y $ instead of $ x \xrightarrow{i\mid \rho_x(i)} y $ when there is no ambiguity, see Figures~\ref{fig:cyclic2} and~\ref{fig:cyclic3} for examples.\\

For this class, the generated  groups are  finite, according to Antonenko and Russyev, but the simplicity of the structure allows us to be more precise: 

\begin{proposition}\label{prop:subgp}
Let $\auta$ be an $n$-state $k$-letter cyclic automaton. Then $\pres{\auta} \leq \GS^n$.
\end{proposition}
\begin{proof}
Let  $\auta = (Q,\Sigma, \delta, \rho )$, with $Q=\{0,1, \hdots n-1\}$ and $|\Sigma| = k$. Consider the action of some state $q$ on a word $\mot{s}=s_0s_1\hdots s_\ell$, $s_i \in \Sigma$. Since the underlying graph is a cycle, up to renaming the states, we have $\delta_i(q)= q+1 \text{ mod }  n$ for all $i \in \Sigma$. So $\rho_q(\mot{s})= \rho_q(s_0)\hdots \rho_{q + \ell \text{ mod } n}(s_\ell) $. Hence $\rho_q$ acts on $\mot{s}$ like the tuple \[(\rho_q, \rho_{q+1 \text{ mod } n}, \hdots, \rho_{q + n-1 \text{ mod } n})\:.\] The same holds for the other states, hence  $\pres{\auta} \leq \GS^n$.
\qed \end{proof}\medskip

\begin{figure}[ht]%
\centering
\begin{minipage}{.41\textwidth}%
	\centering
 \scalebox{1.1}{\begin{tikzpicture}[->,>=latex]
	\node[state,black]  at (0,-1) (1)  {$0$};
	\node[] (c) [right of=1]  {};
	\node[state,black] (2) [right of=c]  {$1$};
	\path
		 (1) edge[bend left] node[above] {$(1,6,4,3)(2,5)$} (2)
		 (2) edge[bend left] node[below] {$(2,3)(4,5,6)$} (1)
		 ;

	\end{tikzpicture}
	}
\caption{A 2-state 6-letter cyclic automaton (generating ${S}_6^2$).}%
\label{fig:cyclic2}
\end{minipage}%
\hspace*{5mm}
\begin{minipage}{.53\textwidth}%
	\centering
		\scalebox{0.9}{
			\begin{tikzpicture}

				\def \n {3}
				\def \radius {1.8cm}
				\def \margin {10} 

				  \node[draw, circle] at ({360/\n * 0}:\radius) {$1$};
				  \node[] at ({360/\n * (1 -1/2)}:1.3*\radius) {$(1,3,2,6,5,4)$};
				  \draw[->, >=latex] ({360/\n * 0+\margin}:\radius) 
					arc ({360/\n * (1 - 1)+\margin}:{360/\n * (1)-\margin}:\radius);
					
				  \node[draw, circle] at ({360/\n * 1}:\radius) {$2$};
				  \node[] at ({360/\n * (2 -1/2)}:0.3*\radius) {$(1,4)(2,5,3,6)$};
				  \draw[->, >=latex] ({360/\n * 1+\margin}:\radius) 
					arc ({360/\n * (2 - 1)+\margin}:{360/\n * (2)-\margin}:\radius);
					
				  \node[draw, circle] at ({360/\n * 2}:\radius) {$0$};
				  \node[] at ({360/\n * (0 -1/2)}:1.3*\radius) {$(1,6,2,5,4,3)$};
				  \draw[->, >=latex] ({360/\n *2+\margin}:\radius) 
					arc ({360/\n * (3 - 1)+\margin}:{360/\n * (3)-\margin}:\radius);

			\end{tikzpicture}
			}
\caption{A 3-state 6-letter cyclic automaton (generating $(A_6\times A_6\times A_6)\rtimes \presc{(1,\pi,\pi)}$ with the notations set below).}%
\label{fig:cyclic3}%
\end{minipage}%
\end{figure}

We have proved that  \[\pres{\auta} = \pres{(\rho_0,\rho_1,,\hdots,\rho_{n-1}), (\rho_1,\hdots, \rho_{n-1},\rho_0), \hdots, (\rho_{n-1},\rho_0,\hdots, \rho_{n-2})}\:.\] In such case, we say that the group is \emph{circularly generated} by the tuple $(\rho_1, \hdots, \rho_n)$ and write \[\pres{\auta} = \presc{(\rho_0,\rho_1,\hdots,\rho_{n-1})}\:,\] where the subscript $c$ stands for circular.\\
The multiplication law in $\pres{\auta}$ can be seen as the usual multiplication of permutations extended componentwise to tuples.

\begin{remark}
 The probability distribution on groups obtained by picking random cyclic automata is different from the one arising by picking random  $n$-tuples of $k$-permutations. For instance, the probability of generating the trivial group is $1/k!^{n}$ in the first case, and $1/k!^{n
^2}$ in the second case.
\end{remark}

From now on we focus on groups that are circularly generated by a tuple of permutations. Likewise the case studied by Dixon, where two permutations were considered, their signatures  impact the generated group. The easiest case arises when all permutations have signature $1$, then no odd permutation can be generated by the tuple, hence $\pres{\auta} \leq \GA^n$. To deal with this we mimic the notion of signature for tuples by defining,  for any tuple $(\ssp_0,\hdots,\ssp_{n-1}) \in \GS^n$, \[ \sgnpi{\ssp_0,\hdots,\ssp_{n-1}} := (\pi^{1-\sgn{\ssp_0}\over 2}, \hdots, \pi^{1-\sgn{\ssp_{n-1}}\over 2})\:,\] where $\pi$ is a transposition of $\GS$.\\

Then we obtain:
\begin{restatable}{proposition}{propxxgpstruct}
\label{prop:gpstruct}
Let $\auta=(Q,\Sigma,\delta, \rho)$ be a   $n$-state $k$-letter cyclic automaton with $ Q= \{0,\hdots,{n-1}\}$. Then \[\pres{\auta} = \presc{(\rho_0,\hdots,\rho_{n-1})} \leq \GA^n \rtimes \presc{ \sgnpi{\rho_0,\hdots,\rho_{n-1}}}\:,\] where $\pi$ is an arbitrary transposition of $\GS$.
\end{restatable}

In the following sections, we prove that, as for Dixon's theorem, the group circularly generated by $(\ssp_0,\hdots, \ssp_{n-1})$ is generically the biggest one, \ie $\GA^n \rtimes \presc{\sgnpi{\ssp_0,\hdots, \ssp_{n-1}}}$.\bigskip

Note that the connectedness hypothesis (the fact that the automaton consists in a single cycle) matters. Indeed for a collection of cycles we  get:
\begin{restatable}{proposition}{propxxunion}
\label{prop:union}
Let $I= \{1,\hdots, m\}$ and  $\auta = \bigsqcup_{i \in I} \auta_i$ be the disjoint union of cyclic automata $\auta_i$, each  with $n_i$ states, $k_i$ letters, and transitions $\{\rho_{i,j}\}_{j < n_i}$. Then, putting $k=\max_i(k_i)$, we have  
\[\pres{\auta} \leq  \GA^{\lcm_I{(n_i)}} \rtimes E\:,\]
 where $E\lesssim (\Z/2\Z)^{\lcm_I{(n_i)}}$ has size at most  $2^u$, with 
\begin{equation}u={\sum_{j =1}^{m} (-1)^{j-1}\sum_{\substack{i_1 < i
_2 < \hdots < i_j}} \gcd({n_{i_1},\hdots,n_{i_j})}}\:.  \label{incl-excl}
\end{equation}

\end{restatable}

\begin{example}
If we have three automata, all  of size  $2$,  then Equation~\eqref{incl-excl} gives \[u=\underbrace{6 - }_\textrm{ size 1 }\underbrace{(2+2+2)}_\textrm{size 2 }\underbrace{+2}_\textrm{size 3 }= 2\:.\]
If the sizes are  $2,3, \text{ and } 5$, we obtain $u=10-(1+1+1)+1 = 8$. Hence we generate groups of size at most $\card{\GA}^{30} \times 2^8 = {k!^{30} \over  2^{22}}$.
\end{example}
\bigskip

This leads us to disprove the following  conjecture from~\cite{KlMaPi}:``the group generated by a  bireversible  $n$-state $k$-letter automaton is either infinite or of order less than or equal to $k!^n$". This conjecture  fails, for instance, for  the disjoint union of the automata from Figure~\ref{fig:cyclic2} and~\ref{fig:cyclic3} , which generates a group of order $34828517376000000 = 6!^6 / 4 > 6!^5$.

\section{The 2-state Case}\label{sec-2case}

In this section we tackle the case of 2-state cyclic automata. Despite  this limitation, most of the combinatorial complexity appears in this case which yet allows us to keep the proof easily readable. \\
In this section, $\auta$ is a   2-state $k$-letter cyclic automaton with $\rho_0 = \ssp$ and~\mbox{$\rho_1 = \ttp$}.\\
The aim of this section is to prove the following :

\begin{theorem}\label{thm:prob}
Let $\auta$ be a 2-state $k$-letter cyclic automaton. Then 

\[\begin{cases} \lim_{k \to \infty} \mathbb{P}\left( \pres{\auta} \simeq \GS \times \GS  \right) &= 1/2 \:,\\ 
 \lim_{k \to \infty} \mathbb{P} \left( \pres{\auta} \simeq  (\GA\times \GA)\rtimes \pres{(\pi,\pi)} \right) & = 1/4 \:,\\
 \lim_{k \to \infty} \mathbb{P}\left( \pres{\auta} \simeq  \GA\times \GA \right)  &= 1/4 \:. 
\end{cases} \]
where $\pi$ is an arbitrary transposition of $\GS$.
\end{theorem}

\begin{remark}
This is not a direct consequence of  Dixon's theorem. Indeed one can apply Dixon's theorem to $\GS \times \GS$ -- considering  $G_D = \pres{(\ssp,\ttp),(\pi,\rho)}$ -- and get that ${G_D} $ is asymptotically isomorphic to $\GS \times \GS $ with probability $3/8$, to $\GS \times \GA$ with probability $3/8$; $(\GA\times \GA)\rtimes \pres{(\pi,\pi)}$ with probability $3/16$, or to $\GA \times \GA$ with probability $1/16$.
\end{remark}
\medskip

In order to determine the generated group, we focus on the maximal subgroup of $\pres{\auta} \leq \GS \times \GS$ where the first coordinate is $e$.\\

The following lemma allows us to conclude in a restricted number of cases:
\begin{restatable}{lemma}{lemxxcoprime}
\label{lem:coprime}
For  $\gcd(\ordr{\ssp}, \ordr{\ttp}) =  1$ we have $\pres{\auta} = \pres{\ssp,\ttp} \times \pres{\ssp,\ttp}$.
\end{restatable}

However it is very unlikely that two permutations have relatively prime orders~\cite{ErTuV}. To prove  Theorem~\ref{thm:prob} we use a theorem of Jordan (see~\cite{Dix}):

\begin{theorem}[Jordan]\label{thm:jordan}
Let $G \leq \GS$ be a primitive group  containing a cycle of prime length $p \leq k-3$. Then $G$ is either the  symmetric group $\GS$ or the alternating group $\GA$.
\end{theorem}

To apply this theorem we prove that we can find on the second coordinate a primitive group which contains a $p$-cycle.\\
The following proposition  restricts the search of a subgroup to the search for a $p$-cycle.\\
We recall that  the conjugacy classes in $\GS$ are formed by the permutations with the same cycle structures. The same holds in $\GA$,  and, if the cycle structure (including the cycles of length 1) consists only of cycles of odd length with no two cycles of same length,  there are two conjugacy classes and we say that the conjugacy class \emph{splits}, otherwise the elements with the same cycle structure form a single conjugacy class~\cite{Sco}. Using the notation $\pi^{\rho} =  \rho^{-1}\pi\rho$ we hqve

\begin{restatable}{proposition}{propxxprimitive}
\label{prop:primitive}
Let $\pi$ be a $p$-cycle of $\GS$ with $p \text{ prime}$ and $k \geq 5 $, then the groups $G_\pi(\GS)= \pres{\pi^\rho~|~ \rho \in \GS}$ and  $G_\pi(\GA)= \pres{\pi^\rho~|~ \rho \in \GA}$ are primitive.
\end{restatable}

The following  easy but useful lemma shows that, if $\GA \leq \pres{\ssp,\ttp}$ and there exists a $p$-cycle $\pi$ such that $(\idp,\pi) \in \pres{\auta}$, then the maximal subgroup of $\pres{\auta}$ where the first component is $e$ is isomorphic to  either $\GA$ or $\GS$.

\begin{restatable}{lemma}{lemxxconj}
\label{lem:conj}
Let $(\idp,\pi)\in \pres{\auta}$. Then for any $\rho \in \pres{\ssp,\ttp}$ we have $(\idp,\pi^\rho) \in \pres{\auta}$.
\end{restatable}

Hence if we find a cycle of prime length less than $k-3$ we will be able to conclude. The following lemma explains how to find such a permutation
and  contains the essence of the proof of Theorem~\ref{thm:prob}. 

\begin{proposition}\label{prop:qcycle}
Let $\ssp$ and $\ttp$ be two permutations of different orders and such that $\pres{\ssp,\ttp}= \GS \text{ or } \GA$.  Then there exists  a $p$-cycle $\pi$ ($p$ prime) such that $(\idp,\pi) \in  \presc{(\ssp,\ttp)}$.
\end{proposition}

\begin{proof}

Put $d= \gcd(\ordr{\ssp},  \ordr{\ttp})$. By construction $\ordr{\ssp^d}$ and $\ordr{\ttp^d}$ are co-prime. Assume $\ordr{\ttp^d} \neq \idp$: choose any  prime number $p$ that divides the order of $\ttp^d$ (hence $p$ does not divide the order of $\ssp^d$). Let~$a$ be the largest integer such that $p^a$ divides $\ordr{\ttp^d}$. Hence $\ttp^{d p^{a-1}}$ has order $p r$, with $\gcd(p, r) = 1$, and  $\ssp^{d p^{a-1}}$ has order co-prime with $p$ (since it is the same order as $\ssp^d$). For simplicity we put $\hat\ttp = \ttp^{d p^{a-1}} \neq \idp $ and  $\hat\ssp = \ssp^{d p^{a-1}}$.\\
Then $ \pres{\auta} \ni(\hat\ssp,\hat\ttp)^{\ordr{\hat\ssp}\ordr{\hat\ttp} \over p}  = ((\hat\ssp^{\ordr{\hat\ssp}})^{\ordr{\hat\ttp}\over p }, (\hat\ttp^{\ordr{\hat\ttp} \over p})^{\ordr{\hat\ssp}}) = (\idp, (\hat\ttp^{\ordr{\hat\ttp} \over p})^{\ordr{\hat\ssp}})$. Hence ${\hat\ttp}^{\ordr{\hat\ttp} \over p}$ has order $p$ by construction and, since $\gcd(\ordr{\hat\ssp}, p) = 1$, so has $(\hat\ttp^{\ordr{\hat\ttp} \over p})^{\ordr{\hat\ssp}} = \check{\tau}$.\\
Since $\check\ttp$ has order $p$, it is a product $\prod_i \pi_i$ of $\ell$ disjoint $p$-cycles. For $\ell > 1$, we construct a $p$-cycle by considering the different cases:

	\begin{enumerate}
		\item Case $p \neq 2$. Let $\tilde\ttp $ be the conjugate $ \pi_1\prod_{i \geq 2}\pi^{-1}_i$  of $\check\ttp$. Hence $\pres{\auta} \ni (\idp,\check\ttp)(\idp,\tilde\ttp)  = (\idp,\pi_1^2)$ and, since $\gcd(2, p) = 1$, $\pi_1^2$ is a $p$-cycle.
		\item Case $p=2$. 
			\begin{enumerate}
				\item If $2\ell < k$. Then $\tilde\tau  = \prod_{i \geq 0} (2i+1,2i+2)$ and $\dot\tau  = (1,k)\prod_{i\geq 0} (2i+1,2i+2) $ are in the same conjugacy class as $\check\tau$, hence $\pres{\auta} \ni (\idp,\tilde\tau)(\idp,\dot\tau) = (\idp,(1,2,k))$, whence we obtain a $3$-cycle.
				\item Otherwise $\tilde\tau  = \prod_{i \geq 0} (2i+1,2i+2)$ and $\dot\tau  = (1,4)(2,3)\prod_{i \geq 2} (2i+1,2i+2) $ satisfy $\tilde\tau \dot\tau =(1,3)(2,4)$ which reduces to the previous case (assuming $k> 4$).
			\end{enumerate}
	\end{enumerate}
\qed \end{proof}

We can now state our theorem:

\begin{theorem}\label{thm:2}
Let $\ssp$ and $\ttp$ be two permutations of $\GS$, $k>5$,  of different orders and such that $\pres{\ssp,\ttp}= \GS \text{ or } \GA$.  Then

\[
\pres{\autpd{0.5}{-.9em}}=  \begin{cases} \GS \times \GS  & \text{ for } \sgn{\ssp} \neq \sgn{\ttp},\\ 
  (\GA\times \GA)\rtimes \pres{(\pi,\pi)}   & \text{ for } \sgn{\ssp} = \sgn{\ttp} = -1,\\
 \GA\times \GA  & \text{ for } \sgn{\ssp} = \sgn{\ttp} = 1,
 \end{cases}
 \]
where $\pi$ is an arbitrary transposition of $\GS$.
\end{theorem}
 \begin{proof}

We apply Proposition~\ref{prop:qcycle}. If the obtained $p$-cycle  satisfies $p\leq k-3$ then we apply Theorem~\ref{thm:jordan} and get that our group is one of the three described above.\\
 Otherwise we can apply the same argument to $\ssp$ and conclude, except in the six cases where (with $d = \gcd{\ordr{\ssp},\ordr{\ttp}}$) $(\ordr\ssp^d, \ordr\ttp^d)$ is  $((k-2)^a,(k-1)^b)$, $((k-1)^a,k^b)$, $((k-2)^a,k^b)$, $(1,(k-2)^b)$, $(1,(k-1)^b)$, or  $(1, k^b)$, and $k=p$ is a prime greater than 5.
  Let us  deal with the first three cases.\\
By taking the suitable powers of $(\ssp,\ttp)^d$,  we can assume $a=b=1$ without loss of generality, and obtain the elements $(\idp,\hat\ssp)$ and $(\idp, \hat\ttp)$ of $\pres{\auta} $ with orders $p-2$ and $p-1$ (\resp $p-1$ and  $p$,  and $p-2$ and $p$). Then if the orders differ by one, we can construct $(\idp,(1,2)) \in \pres{\auta}$ by considering the product $(1,\hdots,p-2)(p-1,p-2,\hdots, 1)$ (\resp $(1,\hdots,p-1)(p,p-1,\hdots, 1)$), hence generating a group isomorphic to the direct product  $\{\idp\}\times \GS$. If the difference is 2 then we use the same trick to get the $3$-cycles. Since 3-cycles generate~$\GA$ we obtain  $\GA \times \GA \leq \pres{\auta}$.\\
In the  last three cases we have $\ssp^d = 1$ , with  $d \leq 2$. For $d=1$,  $\pres{\ssp,\ttp} \neq \GS,\GA$ which contradicts the hypothesis. For $d=2$, we can assume that, up to renaming, $\ttp=(1,2)(3,\hdots,p)$. Hence we can get $(\idp,(1,\hdots,p-3,p-2))$ and $(\idp,(1,\hdots,p-3,p-1))$ in $\pres{\auta}$. Then $(\idp,(1,\hdots,p-3,p-2))(\idp,(1,\hdots,p-3,p-1))=(\idp,(1,3,5,\hdots,p-2)(2,4,\hdots,p-3,p-1))$, so we can find a cycle of prime size smaller than $(p-1)/2$.\\
Finally, as $(\ssp, \ttp) \in \pres{\auta}$, by looking at its signature  we show that the expected group is contained in $\pres{\auta}$. We apply Proposition~\ref{prop:gpstruct}  as an upper bound  and get the result. 
 \qed \end{proof}
 
 To prove Theorem~\ref{thm:prob} we use the convergence of the logarithm of the order of a random permutation to a (continuous) Gaussian limit law~\cite{ErTuIII} to prove that two generic random permutations have different orders  and Dixon's theorem to prove that they generate either $\GA$ or $\GS$.

It would be interesting to compute a more precise asymptotic for this convergence. By Dixon's theorem we know that the probability for two random permutations to generate either $\GS$ or $\GA$ is $1-1/k+O(1/k^2)$. We are now interested in the probability for two random permutations to have the same order. It is clear that this probability admits $1/k^2$ for a lower bound (corresponding to the probability of drawing two $k$-cycles). A more precise lower bound is the probability that the permutations are conjugate, namely
   $W_1/k^2$ where $W_1 \approx 4.26340$ is an explicit constant, see~\cite{FlFuGoPaPo} and~\cite{BlBrWi}. Experimentally the probability that two random permutations have the same order seems to be of the same magnitude, hence we state the following conjecture:
   
\begin{conjecture}\label{conj:sameorder}
Let $\ssp, \ttp$ be two random permutations of $\GS$. Then \[\lim_{k \to \infty }\mathbb{P} \left( \os=\ot \right) = {K \over k^2}\:,\] with $W_1\leq K \leq 12$.
\end{conjecture}

\begin{remark}
The hypothesis ``the permutations have different orders" is sufficient but not necessary: for instance,   $ \ssp = (1,6,7,3,12,5)(2,8)(9,11) $ and $\ttp=(1,9,8)(3,5,7,6,10,11)(4,12)$ satisfy $ \sgn{\ssp,\ttp} = (-1, 1)$ and  $\ordr{\ssp}=\ordr{\ttp}= 6$, but the generated  group is $\GSk{12}\times\GSk{12}$, and other examples can be found for the other  groups described.\\
On the other hand, one can find examples of a cyclic automaton with two permutations with the same order (generating either the symmetric or the alternating group), that generate none of the three expected groups.\\
A necessary and sufficient, but yet ineffective, hypothesis would be that some product $\rho$ of $\ssp$ and $\ttp$  has an order different from the order of ${{\rho}}$ defined by applying $\ssp \mapsto \ttp$ and $\ttp \mapsto \ssp$ on $\rho\in \{\ssp,\ttp\}^*$, and $\GA \leq \pres{\ssp,\ttp}$.
\end{remark}

\begin{remark}
This theorem, likewise Dixon's theorem, does not generalise to semigroups. Indeed to generate the whole transformation semigroup one need to pick at least two permutations, which is rare: the probability of drawing one permutation is $k!/k^k \sim_\infty \sqrt{2\pi k }e^{-k} \to_\infty 0$.
Numerical simulations performed using GAP system~\cite{GAP4} do not suggest any obvious convergence to a set of semigroups.
\end{remark}

Using~\cite{East-Nordahl:j} we can also say a bit about another way of generating random groups:
\begin{proposition}
Let $\ssp$ and $\ttp$ be two random permutations of $\GS$, $k \geq 5$. Let $\auta$ be the disjoint union of two cyclic automata,  with output functions~$\ssp$ and~$\ssp^{-1}$ for one, and  $\ttp$  and $\ttp^{-1}$ for the other. Then, generically,

 \[
\pres{\auta}=  \begin{cases}  (\GA\times \GA)\rtimes \pres{(\pi,\pi)} & \text{ if one of the permutation is odd }\\ 
  
 \GA\times \GA  & \text{ otherwise.}
 \end{cases}
 \]
 With $\pi$ an arbitrary transposition.
\end{proposition}
\begin{proof}
We have $\pres{\auta}= \presc{(\ssp,\ssp^{-1}), (\ttp,\ttp^{-1})}$.
It is a direct consequence of Dixon's theorem ($\pres{\ssp,\ttp}  = \GS, \GA$ generically) and Lemma 2.4 and Example 2.7 of~\cite{East-Nordahl:j}, with the remark that $[\GA,\GA]=\GA$ for $k\geq 5$.

\qed \end{proof}

In contrast,  the group $\presc{(\ssp,\ttp^{-1}),(\ssp^{-1},\ttp)}$ is equal to $\presc{(\ssp,\ttp^{-1})}$, and follows the behaviour described in Theorem~\ref{thm:prob}.

\section{Cyclic Automata with any Number of States}\label{sec-genecase}
For the general case we need more conditions on the orders to find an isolated, prime-sized cycle. In order to conclude, we require that the tuple of the orders of the permutations has no periodic pattern.

Recall that a primitive word is a word that cannot be expressed as a power of a shorter word.
\begin{restatable}{proposition}{propxxqcyclen}
\label{prop:qcyclen}
	Let $(\ssp_i)_i$  be $n$ permutations of $\GS$, $k \geq 7$, such that the tuple $(\ordr{\ssp_0}, \hdots, \ordr{\ssp_{n-1}})$ is primitive and  $\pres{(\ssp_i)_i}= \GS \text{ or } \GA$.  Then there exists a prime-sized cycle $\pi$ satisfying $(\idp,\idp,\hdots,\idp, \pi) \in  \presc{(\ssp_i)_i}$.
\end{restatable}

Now we have: 
\begin{theorem}\label{thm:gene}
Let $\auta_k$ be a  random $n$-state $k$-letter  cyclic automaton with output functions~$(\ssp_0, \hdots, \ssp_{n-1})$.  Then

\[
\lim_{k \to \infty}\mathbb{P} \Big( \pres{\auta_k}=  
  \GA^n \rtimes \presc{\sgnpi{\ssp_0, \hdots, \ssp_{n-1}}}   
\Big) = 1\:.
 \]

\end{theorem}
\begin{proof}
By Dixon's theorem we can assume that $\pres{(\ssp_i)_i}= \GS \text{ or } \GA$.
If the tuple of orders is primitive then using Proposition~\ref{prop:qcyclen} we get an element having  a $p$-cycle on one coordinate and trivial permutations elsewhere, and, as in Theorem~\ref{thm:2} we obtain the expected result. Now the probability that the tuple of orders is primitive is less than the probability that $\ordr{\ssp_0} = \ordr{\ssp_i}$ holds for some $i$, which goes to  $0$ as $k$ goes to infinity, since $n$ is fixed.
\qed \end{proof}

Once again the signatures control the generated group, so we can deduce the probability of generating each group. Note that there is at most $2^{n}$ groups and that some of them are isomorphic.
\section{Conclusion and Perspectives}
\label{sec:ccl}

We showed that the model of random cyclic automaton groups presents little diversity and hence cannot be used in order to provide an efficient model of random groups. The natural extension is now to deal with other types of Mealy automata. The closest step -- the whole class of automata with cycles without exit -- also seems to provide a Dixon-like property. 
One can  prove another analogue to Dixon's Theorem in  two other subclasses: paths ending with a loop and trees with loops on leaves.
\begin{restatable}{proposition}{propxxlinear}
\label{prop:linear}
Let $\auta=(Q,\Sigma,\delta, \rho)$ be an   $n$-state $k$-letter  automaton with $ Q= \{0,\hdots,{n-1}\}$, such that $\delta_u(q)=\min{(q+1, n-1)}, \forall q \in Q, u\in \Sigma$. Then 
\[
\lim_{k \to \infty}\mathbb{P} \big( \pres{\auta_k} =  
  (\GA^{n-1}\times \pres{\rho_{n-1}}) \rtimes P\big) = 1 \:,
 \]
 where $P=\pres{\{(\sgnpi{\rho_i},\hdots, \sgnpi{\rho_{n-1}},\hdots, \sgnpi{\rho_{n-1}})\}_i} $.
\end{restatable}
\begin{restatable}{proposition}{propxxconverging}\label{prop:converging}
Let $\auta=(Q,\Sigma,\delta, \rho)$ be an   $n$-state $k$-letter  automaton with $ Q$  labelled by words in $ \{0,\hdots, a-1\}^{d}$ (with $a$ and $d$ integers respectively called arity and depth of the automaton), such that $\: \delta_u(q_{x_0,\hdots, x_i,x_{i+1}})=q_{x_0,\hdots, x_i}, \forall x_0,\hdots x_{i+1} \in \{0,\hdots, a-1\}^{d}, u\in \Sigma$. Then 
\[
\lim_{k \to \infty}\mathbb{P} \big( \pres{\auta_k}=  
  (\GA^{n-1}\times P \big) = 1  \:,
 \]
  where $P=\pres{\rho_{q_{\epsilon}}}\rtimes \pres{\{(\sgnpi{\rho_{u_{i}u_{i-1}\hdots}},\hdots, \sgnpi{\rho_{q_{\epsilon}}},\hdots, \sgnpi{\rho_{q_{\epsilon}}})\}_i} $.
\end{restatable}

In particular,  if the automaton $\auta$ contains a cyclic automaton as a leaf, then generically $\GA \rtimes \presc{\sgnpi{\ssp_1, \hdots, \ssp_n}}  $ is a subgroup of $\pres{\auta}$, preventing  to generate small groups. This, in addition with numerical evidences, suggest that the class of automata with cycles without exit does not present interesting behaviour regarding random generation.\\
The more general classes of automata are more delicate to tackle since the finiteness problem is not yet known, however they present much more interesting behaviour since we no longer require a tree-like structure.\medskip

Besides, determining the probability that two permutations have the same order, which would provide asymptotics for our result, seems to be an interesting problem  in its own.

\subsection*{Acknowledgements}
I would like to thank Cyril Nicaud for the initial working session that lead to this paper. I am also grateful to my PhD supervisors, Ines Klimann and Matthieu Picantin for their patient listening and the proof reading of this work, and to Charles Paperman for the encouragements and remarks. I am also grateful to the anonymous reviewers for their careful reading and comments.\\
Finally, I am thankful to the many persons I bothered with the question \textit{``what’s the probability that two random permutations have the same orders?”}, in particular Vlady Ravelomanana, Nicolas Pouyanne, Guillaume Lagarde, Stefan-Christoph Virchow, Simon R. Blackburn, John~R.~Britnell, and Mark Wildon – thanks for your help.

\bigskip
\bibliographystyle{amsplain}
\bibliography{Dixon}

\appendix
\section{Appendix}
We present the omitted proofs.  For coherence sake, we recall the statements using the same numerotation. To improve the legibility we add some intermediate results.

\propxxgpstruct*
\begin{proof}
Let $\pi$ be an arbitrary transposition  of~$\GS$. For~$\ssp \in \GS$ , we define 
\[ \hat{\sigma} = \begin{cases}
\{\ssp\pi, \pi \}  & \text{ for } \sgn{\ssp} = -1 \;,\\
\{\ssp, \idp \}  & \text{ otherwise}.
\end{cases}\] The map $\ssp \mapsto \hat{\ssp}$ is clearly a bijection from $\GS $ to $\GA \rtimes \pres{\pi}$. We can extend this bijection to tuples: $\GA^n$ is a normal subgroup of $\GS^n$ as it is a direct product of $n $ normal subgroups of $\GS$, it is also the kernel of the morphism $\text{sgn}_{\pi
}~:~\GS^n~\to~\presc{(\pi, \idp, \idp, \hdots, \idp)}$. Hence $\GS^n \simeq \GA^n \rtimes \presc{(\pi, \idp, \idp, \hdots, \idp)}$ and \begin{align*}
 \pres{\auta} \simeq & \presc{\{( \rho_0\pi^{1-\sgn{\rho_0} \over 2} , \hdots, \rho_{n-1}\pi^{1-\sgn{\rho_{n-1}} \over 2})~\sgnpi{\rho_0,\hdots \rho_{n-1}}  \}}\\
 \leq &\GA^n \rtimes \presc{\sgnpi{\rho_0,\hdots,\rho_{n-1}}}\:.
\end{align*}
\qed \end{proof}

We extend the notion of signature to tuple componentwise:  \[\sgn{\ssp_0,\hdots,\ssp_{n-1}} := (\sgn{\ssp_0},\hdots, \sgn{\ssp_{n-1}})\:.\]
\propxxunion*
\begin{proof}
Similarly to Proposition \ref{prop:subgp}, we show that  every state $q_0^i \in \auta_i$  acts like  $(\rho_{i,1}, \hdots, \rho_{i,{n_i-1}},\rho_{i,0}, \hdots, \rho_{i,{n_i-1}},\hdots \rho_{i,0}, \hdots, \rho_{i,{n_i-1}})$, a tuple of size $\lcm_I(n_i)$ of permutations of $\GS$, hence \[\pres{\auta}\leq \GA^{\lcm_I{(n_i)}} \rtimes \presc{\{\sgnpi{\rho_{i,0}, \hdots, \rho_{i,{n_i-1}}, \hdots\rho_{i,{n_i-1}}}\}_{i \in I}} \leq \GS^n\:.\] 
Consider the elements of $E$: $\presc{ \sgnpi{\rho_{i,0}, \hdots, \rho_{i,{n_i-1}} \hdots\rho_{i,{n_i-1}}}} $ is isomorphic to $\pres{\sgn{\rho_{i,0}, \hdots, \rho_{i,{n_i-1}} \hdots\rho_{i,{n_i-1}}}} \leq \presc{(1,0,\hdots,1,0,\hdots,0 )}$, where the two last groups have additive laws. So the order of~$E$ is less than the order of the group circularly generated by tuples of periods $n_i$ with only  one entry $1$ per period. This group has at most  $2^{\sum n_i}$ elements: put~$E_i  =  \presc{\sgn{\rho_{i,0}, \hdots, \rho_{i,{n_i}} \hdots\rho_{i,{n_i}}}}$ and $P_i = \presc{(1,0,\hdots,1,0,\hdots,0 )}$ (where the tuple has length $\lcm_I (n_i)$, period $n_i$ and one entry~$1$ in each periodic factor), then $\card{E_i} \leq \card{P_i} = 2^{n_i}$ and every element of $E$ is a product of elements in $E_i$, hence $\card{E}\leq \prod_i \card{E_i} \leq \prod \card{P_i} =2^{\sum{n_i}}$. In addition, tuples whose period divides both $n_i$ and $n_j$  belong to both $P_i$ and $P_j$, hence are counted twice. By the inclusion-exclusion principle on the lowest common multiples of the periods we get the result.
\qed \end{proof}

\lemxxcoprime*
\begin{proof}
By Bezout lemma there exist  $u,v \in \N$ such that $u\ordr{\ssp} + v\ordr{\ttp} =  1$.\\
 Now $\pres{\auta} \ni (\ssp,\ttp)^{u\ordr{\ssp}}  =( \ssp^{u\ordr{\ssp}}, \ttp^{1 - v\ordr{\ttp}})=(\idp,\ttp)$. Likewise we get the elements $(\idp,\ssp), (\ssp,\idp)$ and $(\ttp,\idp)$ whence $\pres{\ssp,\ttp} \times \pres{\ssp,\ttp} \leq \pres{\auta}$. The other inclusion  follows from Proposition~\ref{prop:subgp}. 
\qed \end{proof}

\propxxprimitive*
\begin{proof}
We only need to prove the second statement, which implies the first one. Recall that a primitive group is a transitive group that does not stabilise any partition of $\Sigma=\{1, \hdots, k\}$ (apart from the two  trivial ones: the singleton partition and the partition of singletons). Let us prove first that $G_\pi(\GA)$ is transitive. Let $i, j \in\Sigma$.  If the conjugacy class does not split then $(i,j,x_3, \hdots, x_p) = \pi^\rho$ for a suitable $\rho \in \GA$ and arbitrary $x_ \ell$'s, and $\pi^\rho(i)=j$. Otherwise $p \geq 5$ and  if $ (i,j,x_3, \hdots, x_p) = \pi^\rho$,with  $\rho \in \GS \setminus  \GA$ is not in the same conjugacy class as $\pi$ then $ (i,j,x_3, \hdots, x_p,x_{p-1}) = (\pi^\rho)^{(p-1,p)} = \pi^{\rho (p-1,p)}$ is in the same conjugacy class since $\rho{(p-1,p)} \in \GA$. Hence $G_\pi(\GA)$ is transitive.\\
For primitivity, consider a partition $\Sigma_1,\hdots, \Sigma_a$ of $\Sigma$ such that $i,j \in \Sigma_1$ and $\ell\not\in \Sigma_1$. Then consider the cycle $\pi^\rho=(i,j,x_3,\hdots, x_{p-1},\ell)$: as $\pi^\rho(i)=j \in \Sigma_1$ and $\pi^\rho(\ell)=i \in \Sigma_1$,  the partition is not preserved. For $\rho \in\GA$ we get the result. Otherwise $\rho(1,2) \in \GA$, so $(j,i,x_3,\hdots, x_{p-1},\ell) \in G_\pi(\GA)$ and the same holds.
\qed \end{proof}

Let $\dualp{{\color{white}{\ssp}}} : \{\ssp,\ttp\}^* \to \{\ttp,\ssp\}^*$ be the  involutory morphism defined by $\dualp{\ssp} = \ttp$. Then for $\rho \in \{\ssp,\ttp\}^*$,  $(\rho, \dualp{\rho}) \in \pres{\auta}$ (after identifying words on $\{\ssp,\ttp\}$ with permutation in $\GS$).\medskip

\lemxxconj*
\begin{proof}
Since $(\dualp{\rho},\rho) \in \pres{\auta}$ and $\dualp{\rho^{-1}} = \dualp{\rho}^{-1} $, we have $(\dualp{\rho},\rho)^{-1}(\idp,\pi)(\dualp{\rho},\rho) = (\idp,\pi^\rho)$. 
\qed \end{proof}

To prove Proposition~\ref{prop:qcyclen} we first present the case of three permutations, then generalise to an arbitrary number of them.

\begin{restatable}{proposition}{propxxqcyclethree}
\label{prop:qcycle3}
	Let $\ssp_0$, $\ssp_1$ and $\ssp_2$  be three permutations of $\GS$, $k \geq 7$, that do not all have the same order and such that $\pres{\ssp_0,\ssp_1,\ssp_2}$ is either $ \GS$ or $\GA$.  Then there exists a prime-sized cycle $\pi$ satisfying $(\idp,\idp,\pi) \in  \presc{(\ssp_0,\ssp_1,\ssp_2)}$.
\end{restatable}

\begin{proof}
Since the orders are different, there exist  a prime $p$, an integer $c$, and a non-trivial subset of  $\{{\ssp_i}\}_{i \in \{0,1,2\}}$ such that $p^c$ divides each order of the elements in that subset. If it divides  the order of only one permutation, we can apply the argument of Proposition~\ref{prop:qcycle}. Otherwise, we can obtain, by taking the suitable power of a generator, the tuple $(\idp,\ttp_1,\ttp_2)$ (and  $(\ttp_1,\ttp_2,\idp)$ by circularity). As in Proposition~\ref{prop:qcycle} we can  create,   by conjugation and multiplication,  $(\idp,\rho_1,\alpha)$ (\resp $(\beta,\rho_2,\idp)$) for some $p$-cycle $\rho_1$ (\resp $\rho_2$). Moreover we can get that $\ordr{\alpha}$ (\resp  $\ordr{\beta}$) is $p^a$ for some integer $a$ (\resp $p^b$ for some $b$), so we have $(\idp,\rho_1,\alpha_{\rho_1})$ (\resp $(\beta_{\rho_2},\rho_2,\idp)$) for any $p$-cycle $\rho_1$ (\resp $\rho_2$) in a certain conjugacy class under $\pres{\ssp_0,\ssp_1,\ssp_2}$. Then we can multiply these tuples and obtain $(\beta_{\pi},\pi,\alpha_{\pi})$, where $\pi$ is a cycle of size $3$ (\resp $5$ for $p=3$), and the orders of $\alpha_{\pi}$ and $\beta_{\pi}$ are powers of $p$. Then  $(\beta_{\pi},\pi,\alpha_{\pi})^{\ordr{\alpha_{\pi}}\ordr{\beta_{\pi}}}= (\idp,\hat\pi,\idp)$, where $\hat\pi$ is a $3$-cycle (\resp $5$-cycle), thus by conjugation (\resp multiplication and conjugation) we can get all $3$-cycles, and finally able to generate $\GA$. The result follows.
\qed \end{proof}

\propxxqcyclen*
\begin{proof}
Since the tuple of orders is primitive there exists a prime  $p$ and an integer $c$ such that $p^c$ divides every element of a non-trivial subset of $ \{\ordr{\ssp_i}\}_{i \in \{0,\hdots, n-1\}}$. Hence we get a tuple of permutations of orders $p$ or $1$. We can assume that each non-trivial permutation is surrounded by ones: consider $(\alpha_0,\hdots, \alpha_a, \idp, \pi_0,\hdots,\pi_s,\idp,\beta_0,\hdots,\beta_b )$ where the $\pi$'s have orders $p$ and $\pi_1$ is a $p$-cycle, and its permuted-conjugated  $(\beta'_b,\alpha'_0,\hdots, \alpha'_a, \idp, \pi'_0,\hdots,\pi'_s,\idp,\beta'_0,\hdots,\beta'_{b-1} )$  where the $\pi'$'s have orders $p$ and $\pi'_0$ is a $p$-cycle such that $\pi_1\pi'_0$ has order $r$, with $r$ a prime different from $p$, and $\ordr{\alpha_i}=\ordr{\alpha'_i}$ and $\ordr{\beta_i}=\ordr{\beta'_i}$. By multiplying these two tuples we get 
{\scriptsize { \[\Delta\! =\!(\alpha_0\beta'_b,\alpha_1\alpha'_0,\hdots, \alpha_a\alpha'_{a-1}, \alpha'_a,\pi_0,\pi_1\pi'_0,\hdots,\pi_s\pi'_{s-1},\pi'_s,\beta_0, \beta_1\beta'_0,\hdots, \beta_{b-1}\beta'_b)\:.\]}}
We remark that if a permutation $\alpha_i$ is trivial, then the corresponding coordinate in $\Delta$ has order $p$. Then by taking the suitable power we get a tuple $(\gamma_0,\hdots,\gamma_{a-1},\idp,\idp,\rho_0,\hdots,\rho_{s-1},\idp,\idp,\delta_0,\hdots,\delta_{b-1})$, where trivial coordinates remain trivial and every non-trivial permutation has order~$r$, hence, by induction, we construct a tuple of permutations of prime order~$r$, where each non-trivial permutation is surrounded by trivial permutations (note that if the original tuple is formed by one trivial permutation and a block of permutations of orders~$p$, then using techniques similar to the one in Proposition~\ref{prop:qcycle3} we reduce to this situation).This tuple might be non-primitive, but we can construct such a tuple for each prime greater than~$2$.\\
We now consider two cases:  if  $\exists \Pi=(\ssp_0,\hdots,\ssp_{n-1})^d=(\pi_0,\hdots,\pi_{n-1})$ a tuple with $\ordr{\pi_i} \in \{1,p\}$, $p$ prime, such that the tuple of orders is primitive, then  let $ \Gamma$ be a non-trivial tuple with the orders of the coordinate either $1$ or $p$, and where two non-trivial permutations are consecutive (as constructed above). Since $\Pi$ is aperiodic we can find a configuration were two non-trivial permutations of $\Gamma$ face respectively a trivial permutation and a non-trivial one. Thus by multiplying and taking the suitable power, we get  $\Gamma_1$ having the same properties as $\Gamma$ and with strictly less non-trivial permutations. Hence we obtain $\Gamma_{\infty}\in \pres{\auta}$ which contains exactly one $p$-cycle and trivial permutations.\\
If no such $\Pi$ exists then, since the tuple of orders is aperiodic, there are two tuples $\Pi_1, \Pi_2$  having tuples of orders with non-multiple periods $t_1, t_2$ and such that the lowest common multiple of their periods is greater does not divide $n$ (or equivalently is greater than $n$). Then can construct $\Gamma$ as above for $\Pi_1$. Then, either $\Gamma$ is aperiodic or $t_1$ does not divide its period, or $t_1$ divides its period. In the first case we can obtain a tuple $\Gamma_1$ having strictly less non-trivial permutations by multiplying $\Gamma$ and $\Pi_1$ and taking the suitable power. In the second case $t_2$ does not divide the period of $\Gamma$ and we apply the same argument using $\Pi_2$ instead of $\Pi_1$. By induction of this procedure we obtain $\Gamma_{\infty}$ as described above.
\qed \end{proof}

\propxxlinear*
\begin{proof}
By definition of the automaton, $q$ induces the transformation $(\rho_q, \rho_{q+1},\hdots, \rho_{n-1},\hdots, \rho_{n-1},\hdots)$ on $\Sigma^*$. Hence \[\pres{\auta} \simeq  \pres{(\rho_0,\hdots,\rho_{n-1}),(\rho_1,\hdots,\rho_{n-1},\rho_{n-1}),\hdots,(\rho_{n-1},\hdots,\rho_{n-1})}\:,\] whence the result by Dixon's theorem and conjugation.
\qed \end{proof}
Hence the group has size $k!^{n-1}\times \ordr{\rho_{n-1}}/2^{(\min_i{( i | \{\sgn{\rho_{n-1-i}}\}_i = \{-1,1\}}) -1)}$
In the same spirit, if the automaton is the disjoint union of linear automaton, then it is generically the direct product of symmetric or alternating groups.

\propxxconverging*
\begin{proof}
By definition of the automaton, $q_{u_0u_1\hdots u_i}$ induces the transformation $(\rho_{q_{u_0u_1\hdots u_i}}, \rho_{q_{u_0u_1\hdots u_{i-1}}},\hdots, \rho_{q_\epsilon},\hdots, \rho_{q-{\epsilon}},\hdots)$ on $\Sigma^*$. Whence the result by Dixon's theorem and conjugaison.
\qed \end{proof}
\end{document}